\documentclass{amsart}
\usepackage{amsfonts,amsmath,amsthm,amssymb}

\newcommand{\IZ}{\mathbb Z}
\newcommand{\IQ}{\mathbb Q}
\newcommand{\IN}{\mathbb N}
\newcommand{\IR}{\mathbb R}
\newcommand{\la}{\langle}
\newcommand{\ra}{\rangle}

\newcommand{\w}{\omega}

\title[Partitions of matroids into independent subsets]{Partitions of groups and matroids into independent subsets}
\author{Taras~Banakh, Igor~Protasov}
\address{Department of Mathematics, Ivan Franko National University of Lviv, Universytetska 1, 79000 Lviv, Ukraine}
\email{t.o.banakh@gmail.com}
\address{Department of Cybernetics, Kyiv National University, Volodymyrska 64, 01033 Kyiv, Ukraine}
\email{i.v.protasov@gmail.com}

\theoremstyle{plain}
\newtheorem{theorem}{Theorem}
\newtheorem{lemma}{Lemma}

\newtheorem{corollary}{Corollary}
\newtheorem{problem}{Problem}
\newtheorem{claim}{Claim}

\theoremstyle{definition}

\newtheorem{remark}{Remark}

\begin{document}

\begin{abstract}
Can the set $\mathbb{R}\setminus\{0\}$ be covered by countably many linearly (algebraically) independent subsets over the field $\mathbb{Q}$? We use a matroid approach to show that an answer is ''Yes'' under the Continuum Hypothesis, and "No" under its negation.
\end{abstract}
\maketitle

In this paper we discuss partitions of groups and other algebraic objects into independent subsets. The notion of an independent set can be defined for any hull operator.

By a {\em hull operator} on a set $X$ we understand any function  $\la\cdot\ra:\mathcal P_X\to\mathcal P_X$ defined on the power-set $\mathcal P_X$ of $X$ such that $A\subset \la A\ra\subset\la B\ra$ for any subsets $A\subset B$ of $X$.

A subset $A\subset X$ is called {\em independent} (with respect to the  hull  operator $\la\cdot\ra$) if $a\notin\la A\setminus\{a\}\ra$ for any point $a\in A$. 

We shall say that a hull operator $\la\cdot\ra:\mathcal P_X\to\mathcal P_X$ on a set $X$ 
\begin{itemize}
\item is {\em idempotent} if $\la\la A\ra\ra=\la A\ra$ for each subset $A\subset X$;
\item has {\em finite supports} if for each $A\subset X$ and $a\in\la A\ra$ there is a finite subset $F\subset A$ with $a\in\la F\ra$;
\item has {\em the MacLane-Steinitz exchange property} if for any subset $A\subset X$ and points $x,y\in X\setminus\la A\ra$ the inclusion $x\in\la A\cup\{y\}\ra$ is equivalent to $y\in\la A\cup\{x\}\ra$; 
\end{itemize}

By a {\em matroid} we understand a pair $(X,\la\cdot\ra)$ consisting of a set $X$ and an idempotent hull operator $\la \cdot\ra:\mathcal P_X\to\mathcal P_X$ that has finite supports and the MacLane-Steinitz exchange property. For finite $X$ our definition of matroid argees with the classical one, see \cite[1.4.4]{Ox} or \cite[p.~270]{b1}.

\begin{theorem}{\label{t1}} Let $\kappa$ be an infinite cardinal, $X$ be a set of cardinality $|X|=\kappa^+$ and $\la\cdot\ra:\mathcal P_X\to\mathcal P_X$ be a hull operator with finite supports and the MacLane-Steinitz exchange property. If $|\la F\ra\setminus\la\varnothing\ra|\le\kappa$ for each finite subset $F\subset X$, then the set $X\setminus\la\varnothing\ra$ can be covered by  $\kappa$ many independent subsets.
\end{theorem}

\begin{proof} If $|X\setminus\la \varnothing\ra|\le\kappa$, then $X\setminus\la\varnothing\ra$ can be decomposed into $\le\kappa$ many singletons (which are independent sets). So, we assume that $|X\setminus\la\varnothing\ra|=\kappa^+$.

The finite support property of the hull operator and Zorn's Lemma imply the existence of a maximal independent subset $A\subset X$. We claim that $X=\la A\ra$. Assuming the opposite, we can find a point $x\in X\setminus\la A\ra$ and consider the set $A\cup\{x\}$ which is dependent as $A$ is maximal independent. Consequently, there is a point $a\in A\cup\{x\}$ such that $a\in\la (A\cup\{x\})\setminus\{a\}\ra$.
This point $a$ is not equal to $x$ as $x\notin\la A\ra$. Let $B=A\setminus\{a\}$. Since $A$ is independent, $a\notin\la A\setminus\{a\}\ra=\la B\ra$. Since the hull operator $\la \cdot\ra$ has the MacLane-Steinitz exchange property and $x,a\notin\la B\ra$, the inclusion $a\in\la B\cup\{x\}\ra$ implies $x\in\la B\cup\{a\}\ra=\la A\ra$, which contradicts the choice of $x$. Therefore $X=\la A\ra$. 

\begin{claim}\label{cl1} $|\la B\ra\setminus\la\varnothing\ra|\le\max\{|B|,\kappa\}$ for any subset $B\subset X$.
\end{claim}

\begin{proof} Let $[B]^{<\w}$ denote the family of finite subsets of $B$. Since the hull operator $\la\cdot\ra$ has finite supports, $\la B\ra=\bigcup_{F\in[B]^{<\w}}\la F\ra$. By our assumption, $|\la F\ra\setminus\la\varnothing\ra|\le\kappa$ for each finite subset $F\subset X$. Consequently,
$$
\begin{aligned}
|\la B\ra\setminus\la\emptyset\ra|&=\Big|\bigcup_{F\in[B]^{<\w}}\la F\ra\setminus\la\varnothing\ra\Big|\le\\
&\sum_{F\in[B]^{<\w}}|\la F\ra\setminus\la\varnothing\ra|\le\max\{\kappa,|[B]^{<\w}|\}\le\max\{\kappa,|B|\}.
\end{aligned}$$
\end{proof}

By Claim~\ref{cl1}, $\kappa^+=|X\setminus\la\varnothing\ra|=|\la A\ra\setminus\la\varnothing\ra|\le\max\{\kappa,|A|\}$ and thus $|A|=\kappa^+$. Fix an injective enumeration $A=\{a_\alpha:\alpha<\kappa^+\}$.

For every ordinal $\alpha<\kappa^+$,  put
$$A_\alpha=\{a_\gamma:\gamma<\alpha\}
\mbox{ \ and \ } X_\alpha=\la A_{\alpha+1}\ra\setminus \la A_\alpha\ra.$$

Then $X\setminus\la \varnothing\ra=\bigcup_{\alpha<\kappa^+}X_\alpha$ and 
for every $\alpha<\kappa^+$ $$|X_\alpha|=|\la A_{\alpha+1}\ra\setminus\la A_\alpha\ra|\le|\la A_{\alpha+1}\ra\setminus\la\varnothing\ra|\leqslant\kappa$$ according to Claim~\ref{cl1}. For every $\alpha<\kappa^+$  fix an injection $\chi_\alpha:X_\alpha\to\kappa$ and define a function $\chi:X\setminus\la \varnothing\ra\to\kappa$ letting $\chi|{X_\alpha}=\chi_\alpha$ for all $\alpha<\kappa^+$. For every $\lambda<\kappa$ consider the set $Z_\lambda=\chi^{-1}(\lambda)$. We claim that $\{Z_\lambda:\lambda<\kappa\}$ is a desired partition into $\kappa$ many independent subsets.

 Assuming that some set $Z_\lambda$ is dependent, we can find a finite dependent subset $F\subset Z_\lambda$. We can assume that $F$ is a minimal dependent subset. Since $F$ is dependent, there is an element $a\in F$ such that $a\in \la F\setminus\{a\}\ra$. We claim that $x\in\la F\setminus\{x\}\ra$ for any element $x\in F$. This is clear if $x=a$. So, we assume that $x\ne a$. Consider the set $E=F\setminus\{x,a\}$. Since $F$ is minimal dependent, the subsets $E\cup\{a\}$ and $E\cup\{x\}$ are independent and thus $a,x\notin\la E\ra$. Since $a\in \la F\setminus\{a\}\ra=\la E\cup\{x\}\ra$ the MacLane-Steinitz exchange property of the hull operator $\la \cdot\ra$ guarantees that $x\in\la E\cup\{a\}\ra=\la F\setminus\{x\}\ra$.

Since $\chi$ is injective on each subset $X_\alpha$, there exist a numeration $F=\{f_1,\ldots,f_n\}$ and ordinals $\alpha_1<\alpha_2<\ldots<\alpha_n<\kappa^+$ such that $f_1\in Y_{\alpha_1},\ldots,f_n\in Y_{\alpha_n}$. Then $f_n\notin\la  \{f_1,\ldots,f_{n-1}\}\ra=\la F\setminus\{f_n\}\ra$, which is a desired contradiction. 
\end{proof}

\begin{corollary}{\label{c1}}
Let $\kappa$ be an infinite cardinal, $X\setminus\{0\}$ be a vector space over a field $F$ such that $|F|\leqslant\kappa$ and $\dim X=\kappa^+$. Then $X\setminus\{0\}$ can be partitioned in $\kappa$ linearly independent subsets.
\end{corollary}

\begin{corollary}{\label{c2}}
Let $\kappa$ be an infinite cardinal, $X\setminus\{0\}$ be an extension of a field $F$ such that $|F|\leqslant\kappa$ and $|X|=\kappa^+$. Then $X\setminus\{0\}$ can be partitioned in $\kappa$ subsets algebraically independent over $F$.
\end{corollary}

\begin{corollary}{\label{c3}}
Let $\kappa$ be an infinite cardinal, $K_{\kappa^+}$ be a complete graph with $\kappa^+$ many vertices. Then there exists an edge coloring of $K_{\kappa^+}$ in $\kappa$ colors with no monochrome cycles.
\end{corollary}

Now we consider independent subsets in groups. We define a subset $A$ of a group $G$ to be {\em independent} if $A$ contains no point $a\in A$ with $a\in\la A\setminus \{a\}\ra$ where $\la A\setminus\{a\}\ra$ is the subgroup of $G$ generated by the set $A\setminus\{a\}$. Thus $A$ is independent with respect to the hull operator $\la\cdot\ra:\mathcal P_G\to\mathcal P_G$ assigning to each subset $B\subset G$ the subgroup $\la B\ra$ generated by $B$. This hull operator is idempotent and has finite supports, but in general fails to have the MacLane-Steinitz exchange property. Because of that, Theorem~\ref{t1} is not applicable to the following open problem.

\begin{problem} Let $G$ be an (abelian) group $G$ of cardinality $|G|=\aleph_1$ with neutral element $\{e\}$. Can $G\setminus \{e\}$ be covered by countably many independent subsets?
\end{problem}

Let us prove that for groups of cardinality $>\aleph_1$ this problem has a negative solution.

\begin{theorem}{\label{t2}}
Let $\kappa$ be an infinite cardinal and $G$ be a group of cardinality $|G|>\kappa^+$. Then, for every coloring $\chi: G\setminus\{e\}\to\kappa$, there exists a dependent monochrome subset $A$ of $G$ of cardinality $|A|=4$.
\end{theorem}

We shall need a simple combinatorial lemma.

\begin{lemma}{\label{l1}}
Let $\lambda\ge 1$ be a cardinal and $\kappa$ be an infinite cardinal.
Let $X,Y$ be sets of cardinality $|X|=\kappa^+$ and $|Y|>(\kappa^+)^\lambda$. For any $\kappa$-coloring $\chi:X\times Y\to \kappa$ there are subsets $A\subset X$ and $Z\subset Y$ such that $|A|=\lambda$, $|Z|>(\kappa^+)^\lambda$ and the set $A\times Z$ is monochrome. 
\end{lemma}
\begin{proof}
Since $|X|>\kappa$, for every $y\in Y$, there exist a subset $A(y)\subset X$ of cardinality $|A(y)|=\lambda$ such that $A(y)\times \{y\}$ is monochrome and hence $\chi(A(y)\times\{y\})=\{\widetilde\chi(y)\}$ for some color $\widetilde\chi(y)$. Now consider the function
 $f:Y\to[X]^\lambda\times\kappa$ defined by
$$f(y)=(A(y),\widetilde\chi(y)).$$
Since $|[X]^\lambda\times\kappa|\le(\kappa^+)^\lambda<|Y|$ there exists a pair $(A,\alpha)\in[X]^\lambda\times\kappa$ such that the preimage $Z=f^{-1}(A,\alpha)$ has cardinality $|Z|>(\kappa^+)^\lambda$.
Then $A\times Z$ is a required monochrome rectangle.
\end{proof}

\begin{proof}[Proof of Theorem \ref{t2}] Since $|G|>\kappa^+$, we can choose two sets $X,Y\subset G\setminus\{e\}$ of cardinality $|X|=\kappa^+$ and $|Y|>\kappa^+$ such that $Y\cap X^{-1}=\emptyset$. Given any $\kappa$-coloring $\chi:G\setminus\{e\}\to\kappa$, consider the coloring $\tilde\chi:X\times Y\to \kappa$ defined by $\tilde \chi(x,y)=\chi(xy)$. By Lemma~\ref{l1}, there are a 2-element set $A=\{a,b\}\subset X$ and a subset $Z\subset X$ of cardinality $|Z|>\kappa^+$ such that the set $A\times Z$ is monochrome. Choose any point $x\in Z$ and any point $y\in Z\setminus\{a^{-1}bx,b^{-1}ax\}$. Then the set $B=\{ax,bx,ay,by\}\subset G$ has cardinality $|B|=4$ and is monochrome with respect to the coloring $\chi$. Since $ax=ay(by)^{-1}bx$, the set $B$ is dependent.
\end{proof}

Next, we consider covers of abelian groups by linearly independent subsets. Following \cite[\S16]{Fu} or \cite[\S4.2]{Rob}, we call a subset $A$ of an abelian group $G$ {\em linearly independent} if for any pairwise distinct points $a_1,\dots,a_n\in A$ and integer numbers $\lambda_1,\dots,\lambda_n\in\IZ$ the equality
$$\lambda_1a_1+\cdots+\lambda_na_n=0$$ implies $\lambda_1a_1=\cdots=\lambda_na_n=0$.

Observe that a subset $A$ of an abelian group $G$ is linearly independent if and only if it is independent with respect to the hull operator
$[\,\cdot\,]:\mathcal P_G\to\mathcal P_G$ assigning to each subset $B\subset G$ the subset
$$[B]=\{0\}\cup\{x\in G:\exists n\in\IN\;\;nx\in\la B\ra\setminus\{0\}\}.$$The hull operator $[\,\cdot\,]$ is idempotent and has finite supports and the MacLane-Steinitz exchange property. So, the pair $(G,[\,\cdot\,])$ is a matroid.

Since $\la A\ra\subset[A]$ for any $A\subset X$, each independent subset of $G$ is linearly independent. On the other hand, the subset $\{2,3\}$ of the group of integers $\IZ$ is independent but not linearly independent. 

For an abelian group $G$ and a natural number $n\in\IN$ consider the subgroup
$$G[n]=\{x\in G:nx=0\}.$$The union $G[\IN]=\bigcup_{n\in\IN}G[n]$ is called the {\em torsion part} of $G$.
   
\begin{theorem}\label{t3} Let $\kappa$ be an infinite cardinal and $G$ be an abelian group. The set $G\setminus\{0\}$ can be covered by $\kappa$ many linearly independent subsets if and only if $|G|\le\kappa^+$ and either $|G[\IN]|\le\kappa$ or $G=G[p]$ for some prime number $p$.
\end{theorem}
  
\begin{proof} First we prove the ``if'' part. Assume that $|G|\le\kappa^+$ and either $|G[\IN]|\le\kappa$ or $G[p]=G$ for some prime number $p$.

If $|G|\le\kappa$, then $G\setminus\{0\}$ decomposes into $|G\setminus\{0\}|\le\kappa$ many singletons, which are linearly independent.
So, we assume that $|G|=\kappa^+$. 

Consider the hull operator $[\,\cdot\,]:\mathcal P_G\to\mathcal P_G$   turning $G$ into a matroid whose independent sets coincide with linearly indepenent subsets of $G$.
Theorem~\ref{t1} will imply that $G\setminus\{0\}$ can be covered by $\kappa$ many linearly independent subsets as soon as we prove that for each finite subset $F\subset G$ the hull $[F]$ has cardinality $|[F]|\le\kappa$. 

First we consider the case $|G[\IN]|\le\kappa$.
For every $n\in\IN$ consider the homomorphism $n:G\to G$, $n:x\mapsto nx$, whose kernel coincides with the subgroup $G[n]$ that has cardinality $|G[n]|\le|G[\IN]|\le\kappa$. Taking into account that $|\la F\ra|\le\aleph_0$ and $$[F]=\{0\}\cup\bigcup_{n\in\IN}n^{-1}(\la F\ra\setminus\{0\})\subset \bigcup_{n\in\IN}n^{-1}(\la F\ra),$$ we conclude that 
$$|[F]|\le\sum_{n\in\IN}|n^{-1}(\la F\ra)|\le\sum_{n\in\IN}|G[n]|\cdot|\la F\ra|\le \kappa.$$  

Next, assume that $G=G[p]$ for a prime number $p$. Then $G$ is a linear space over the $p$-element field $\mathbb F_p$ and by Corollary~\ref{c1}, $G$ can be covered by $\kappa$ many linearly independent subsets.
\smallskip

To prove the ``only if'' part, assume that $G\setminus\{0\}$ can be covered by $\kappa$ many linearly independent subsets. By Theorem~\ref{t2}, $|G|\le\kappa^+$. It remains to prove that $|G[\IN]|>\kappa$ implies $G[p]=G$ for a prime number $p$. For every prime number $p$, consider the subgroup $G[p^\infty]=\bigcup_{k\in\IN}G[p^k]$. By \cite[4.1.1]{Rob}, $G[\IN]=\operatornamewithlimits{\oplus}\limits_{p\in\Pi}G[p^\infty]$ where $\Pi$ denotes the set of all prime numbers. Since $|G[\IN]|=\kappa^+$, there exists a prime number $p$ such that $|G[p^\infty]|=\kappa^+$ and hence $|G[p^n]|=\kappa^+$ for some  $n\in\IN$. We can assume that $n$ is the smallest number with this property, i.e., $|G[p^{n-1}]|<\kappa^+$.

 We claim that $G=G[p^n]$. Otherwise, we can choose a point $a\in G\setminus G[p^n]$. By our assumption, $G\setminus\{0\}$ can be covered by $\kappa$ many linearly independent subsets. The coset $a+G[p^n]$ has cardinality $\kappa^+$ and hence contains two distinct points $a+x$, $a+y$ that lie in a linearly independent subset of $G$. On the other hand, the  set $\{a+x,a+y\}$ is linearly dependent since $p^n(a+x)-p^n(a+y)=0$ but $p^n(a+x)=p^na\ne0\ne p^n(a+y)$. This contradiction shows that $G=G[p^n]$, which means that the group $G$ has bounded height and then  by Pr\"ufer-Baer Theorem~4.3.5 in \cite{Rob}, $G$ is a direct sum $\oplus_{\alpha<\kappa^+}H_\alpha$ of cyclic $p$-groups of order $|H_\alpha|\le p^k$. Each group $H_\alpha$ contains a cyclic subgroup $C_\alpha$ of order $p$. Then the subgroup $\oplus_{\alpha<\kappa^+}C_\alpha\subset G[p]$ has cardinality $\kappa^+$ and hence $p^n=p$ by the choice of $n$. Therefore, $G=G[p^k]=G[p]$.
\end{proof}

Theorem~\ref{t3} implies that for the direct sum $G=\oplus^{\aleph_1}C_4$  of $\aleph_1$ many cyclic groups of order 4 the set $G\setminus\{0\}$  cannot be covered by countably many linearly independent sets. The same concerns the group  $\oplus^{\aleph_1}C_6=(\oplus^{\aleph_1}C_2)\oplus(\oplus^{\aleph_1}C_3)$. 

\begin{problem} Can the groups $\oplus^{\aleph_1}C_4$ and $\oplus^{\aleph_1}C_6$ with removed zeros be covered by countably many independent sets.
\end{problem}

Corollaries~\ref{c1}, \ref{c2} and Theorem~\ref{t2} imply the following interesting equivalence.

\begin{theorem} The following statements are equivalent:
\begin{enumerate}
\item For the linear space $\IR$ over the field $\IQ$ there is a partition of $\IR\setminus\{0\}$ into countably many linearly independent subsets over $\IQ$.
\item For the extension $\IR$ of the field $\IQ$ there is a partition of $\IR\setminus\{0\}$ into countably many algebraically independent subsets over $\IQ$.
\item For the abelian group $\IR$ there is a partition of $\IR\setminus\{0\}$ into countably many independent subsets.
\item The Continuum Hypothesis is true.
\end{enumerate}
\end{theorem}

\begin{remark}{\label{r2}}
Let $\kappa,\mu$ be infinite cardinals such that $\kappa^+<\mu$. It follows from Lemma~\ref{l1} that for every $\kappa$-coloring of the set of edges of the complete graph $K_\mu$, there exists a monochrome cycle of length 4.
Moreover, S.~Slobodianyuk observed that the proof of Lemma~\ref{l1} gives a monochrome complete bipartite subgraph $K_{\lambda,\mu}$ provided that $\lambda\leqslant\kappa^+$ and $(\kappa^+)^\lambda<\mathrm{cf}\mu$. In particular, there exists a monochrome cycle of any even length.

On the other hand, for the complete graph $K_{2^\kappa}$ its set of vertice $V=2^\kappa$ can be identified with the set of all functions $f:\kappa\to \{0,1\}$, which allows us to define a coloring $\chi:[V]^2\to\kappa$ by the rule
$$\chi(\{f,g\})=\min\{\alpha:f(\alpha)\ne g(\alpha)\}.$$
For this coloring, $K_{2^\kappa}$ has no monochrome cycles of odd length.
\end{remark}

{\bf Acknowledgement.} We would like to thank Sergiy Slobodianuyk for  stimulating questions and remarks.


\begin{thebibliography}{1}

\bibitem{Fu} L.~Fuchs, {\em Infinite abelian groups. I}, Academic Press, New York-London, 1970.

\bibitem{Ox} J.~Oxley, {\em Matroid Theory}, Oxford Univ. Press., Oxford, 1992. 

\bibitem{Rob} D.~Robinson, {\em A course in the theory of groups}, Springer-Verlag, New York, 1996.

\bibitem{b1}
K.~A.~Rybnikov, {\em Introduction to Combinatorial Analysis}, Moskow University Press, 1985.
\end{thebibliography}
\end{document}